\documentclass[10pt,letterpaper]{article}
\usepackage[top=0.85in,left=2.75in,footskip=0.75in,marginparwidth=2in]{geometry}

\usepackage[utf8]{inputenc}

\usepackage{amsmath}   
\usepackage{amsfonts}  
\usepackage{amsthm}    
\usepackage{amssymb}   

\usepackage{cite}

\usepackage{nameref,hyperref}

\usepackage[right]{lineno}

\usepackage{microtype}
\DisableLigatures[f]{encoding = *, family = * }

\usepackage{changepage}

\usepackage[aboveskip=1pt,labelfont=bf,labelsep=period,singlelinecheck=off]{caption}

\makeatletter
\renewcommand{\@biblabel}[1]{\quad#1.}
\makeatother

\usepackage{lastpage,fancyhdr,graphicx}
\usepackage{epstopdf}
\fancyheadoffset[L]{2.25in}
\fancyfootoffset[L]{2.25in}

\usepackage{color}

\definecolor{Gray}{gray}{.25}

\usepackage{graphicx}

\usepackage{sidecap}

\usepackage{wrapfig}
\usepackage[pscoord]{eso-pic}
\usepackage[fulladjust]{marginnote}
\reversemarginpar

\newtheorem{thm}{Theorem}[section]
\newtheorem{theorem}[thm]{Theorem}
\newtheorem{lemma}[thm]{Lemma}

\newtheorem{definition}[thm]{Definition}
\newtheorem{proposition}[thm]{Proposition}
\newtheorem{corollary}[thm]{Corollary}

\begin{document}
\vspace*{0.35in}

\begin{flushleft}
{\Large
\textbf\newline{A characterization of Fibonacci numbers}
}
\newline
\\
Giuseppe Pirillo
\\
\bigskip
Dipartimento di Matematica ed Informatica U. Dini\\
                Università di Firenze\\
                viale Morgagni 67/A\\
                50134 Firenze Italia
\\
\bigskip
* pirillo@math.unifi.it

\end{flushleft}

\section*{Abstract}
The link between
the equation $b(b+a)-a^2=0$ 
concerning the side $b$ and the diagonal $a$
of a regular pentagon
and the {\it Cassini identity}
$F_{i}F_{i+2}-F_{i+1}^2=(-1)^{i}$,
concerning three consecutive Fibonacci numbers,
is very strong.
In this paper we present our thesis that
the two mentioned equations
were  ``almost simultaneously'' discovered
by the {\it Pythagorean School}.

\section*{Introduction}
Let
$F_{0}=1$,
$F_{1}=1$ and, for $n\ge 2$,
$F_{n}=F_{n-2}+F_{n-1}$
be the {\it Fibonacci numbers}.
It is well known that
$\lim_{n \rightarrow \infty}\frac{F_{n+1}}{F_{n}}=\Phi =\frac{1+\sqrt 5}{2}$
and that
in theoretical computer science
the {\it Fibonacci word}
$f=101101011011010110 \dots$
is a {\it cutting sequence} representing
the {\it golden ratio} $\Phi$
(also called {\it Divina Proportione} by Luca Pacioli).
Concerning the Fibonacci numbers,
the Fibonacci word and the golden ratio, see
\cite{ArnouxSiegel},
\cite{Brother},
\cite{Fibonacci},
\cite{Irem},
\cite{Knuth68},
\cite{MR0451916},
\cite{GKP},
\cite{Matiyasevich0},
\cite{Matiyasevich1},
\cite{1997a},
\cite{pirillo2001},
\cite{2005a},
\cite{2005b},
\cite{SEABM},
\cite{Archimede}
\cite{origine}
and
\cite{wasteels}.

It is also well known that,
given three consecutive Fibonacci numbers
$F_{i}\le F_{i+1}<F_{i+2}$,
the following {\it Cassini identity}
$F_{i}F_{i+2}-F_{i+1}^2=(-1)^{i}$
holds.
In this paper we support our thesis
that
the discovery of incommensurability
and of the previous equalities
came ``almost simultaneously'',
most likely
first the pythagorean equality and
immediately after the Cassini identity.

Indeed the Cassini identity is strictly related
to the studies and the fundamental results
of the Pythagorean School (hereafter simply School)
on the incommensurability:
{\it side and diagonal
of the regular pentagon
are incommensurable}
(see Figure \ref{pentagon}).
The result:
{\it if $b$ is the side and
$a$ is the diagonal of a regular pentagon,
then
$b: a = a : (b+a)$
and
$b(b+a)-a^2=0$}
precedes of a very short period of time
the discoveries of Fibonacci numbers
and Cassini identity
$F_{i}F_{i+2}-F_{i+1}^2=(-1)^{i}$
see \cite{SEABM}.

\section*{The irrational number $\Phi$}

The School tried for a long time
to find a common measure
between the diagonal and the side of the regular pentagon.
In the proof of these fundamental results
(that we shortly recall hereafter)
the following
{\it Pythagorean Proposition}
\ref{Pythagorean Proposition}
(see \cite{2005b}) plays a crucial role
(and the same will happen
in the first proof
of the main result of this paper,
Proposition \ref{HippasusisFibonacci}).

\begin{proposition}\label{Pythagorean Proposition}
({\bf Pythagorean Proposition.})
A strictly decreasing sequence of positive
integers is necessarily finite.
\end{proposition}

A common measure of diagonal
and side of a regular pentagon implies
the existence of a
segment $U$
and two positive integers $\beta$ and $\alpha $
such that $U$ is contained $\beta$ times in $b$,
the side,
and $\alpha $ time in $a$,
the diagonal.
Using elementary results on similar triangles,
we easily reach the equalities
$\beta : \alpha  = \alpha  : (\beta +\alpha )$
and
$\beta (\beta +\alpha )=\alpha ^2$.

But, two such integers $\beta$ and $\alpha$
do not exist by an old
well-known odd-even argument:
i) $\beta$ and $\alpha$ both odd
implies
$\beta (\beta + \alpha)$ even
and
$\alpha ^2$ odd
(contradiction),
ii) $\beta$ odd and $\alpha$ even 
implies
$\beta (\beta + \alpha)$ odd
and
$\alpha ^2$ even
(contradiction),
iii) $\beta$ even and $\alpha$ odd
implies
$\beta (\beta + \alpha)$ even
and
$\alpha ^2$ odd
(contradiction),
iv) $\beta$ and $\alpha$ both even then,
using the Pythagorean Proposition
\ref{Pythagorean Proposition},
we retrieve one of the three previous cases
i), ii) and iii)
(contradiction).
So
$\beta$
and
$\alpha$
cannot be both integers.
So {\it side} and {\it diagonal}
of the {\it regular pentagon}
cannot have a common measure
and the following theorem is proved.

\begin{figure}[!hbp]
\centering
\includegraphics[height=6cm, width=5cm, angle=0,
keepaspectratio]{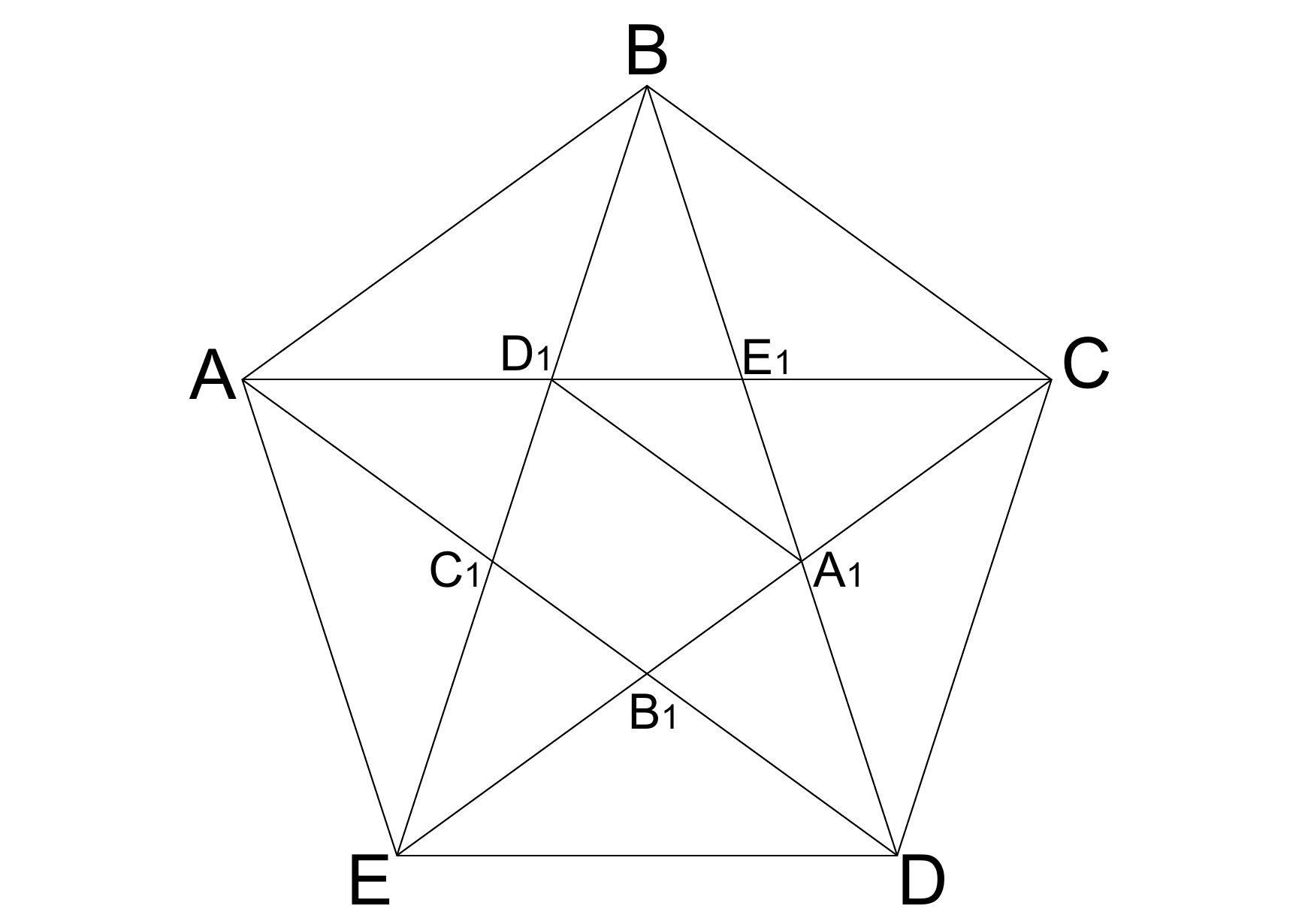}
\caption{Regular pentagon}
\label{pentagon}
\end{figure}

\begin{theorem}\label{theorem}
Side and diagonal
of the regular pentagon
are incommensurable.
\end{theorem}

\section*{Fibonacci numbers and their relation
with incommensurability}

We will present hereafter
an argument that shows how
the Fibonacci numbers and
the Cassini identity
appeared naturally
during the
development of the argument of the
incommensurability.
Several
attempts
to find a common
measure of
side and diagonal
of the regular pentagon
were not successful
and
will hereafter be examined in depth.
Consider two Propositions on the
triangle well known today
and also well known to the School:

\begin{proposition}\label{proposition1}
The greatest side of a triangle is that opposite
to the greatest angle.
\end{proposition}

\begin{proposition}\label{proposition2}
The sum of two sides is greater than the third side.
\end{proposition}

Considering the isosceles triangle
formed by two consecutive sides
and by a diagonal
of a regular pentagon,
the School would have noticed,
by Proposition \ref{proposition1},
the inequality
$\beta <\alpha $
and,
by Proposition \ref{proposition2},
the inequality
$\alpha <2\beta $.
This is enough to immediately eliminate
the side as a common measure
($\beta =1$).

Now, let $\beta \ge 2$.
Being $\beta$ and $\alpha$ integers, from

\centerline{$\beta <\alpha < 2\beta$,}

\noindent we have

\centerline{$\beta + 1 \le \alpha \le 2\beta -1$.}

Considering the necessary equality
$\beta (\beta +\alpha )=\alpha ^2$
and using the above lower bound and upper bound,
the School easily eliminated
the following segments
as common measure:
the half of the side
($2(2+3)-3^2\not= 0$),
the third of the side
($3(3+4)-4^2\not= 0$ and $3(3+5)-5^2\not= 0$),
the fourth part of the side
($4(4+5)-5^2\not= 0$,
$4(4+6)-6^2\not= 0$
e
$4(4+7)-7^2\not= 0$)
and so on.

On the other hand,
continuing in this way
the calculation is increasingly long and difficult
as, for each $\beta >1$,
one must consider
$\beta -1$ candidates for $\alpha$.
The departing geometric problem
(find a common measure $U$)
is now an arithmetic problem:
{\it given an integer $\beta $
does there exist
an integer $\alpha \ge \beta $
such that $\beta (\beta +\alpha )-\alpha ^2=0$?}

When the recalled argument of incommensurability
was completed and consequently it was clear
that the answer to this question would be ``NO''
for each $\beta$,
we believe that the School
has considered the just obtained result
as a motivation for a new research
and has been argumented as follows:
as $\beta (\beta +\alpha )-\alpha ^2$
is never $0$, we wish to see for what values
of $\beta$ and $\alpha$
the difference between the greatest and the smallest
of the numbers $\beta (\beta +\alpha )$ and $\alpha ^2$
assumes the value 1,
which is the minimum possible one.
This is a typical curiosity of mathematicians:
when they solve a problem,
their attention is immediately attracted
by the new and often numerous problems
that the solution always carries with it.
So, we simply believe that,
after the discovery of the incommensurability,
the School has focused
on this new problem.

Today, to find the above recalled values of
$\beta$ and $\alpha$
is very easy
using a computer.
It is possible to write a program
that searches, finds
and puts all these values in the following table.
My brother Mario wrote the program
and this is what happens:

\bigskip

\centerline{
\begin{tabular}{|c|c|c|c|c|}
\hline
$\beta $ & $\alpha $ & $\alpha +\beta $ & $\beta (\alpha +\beta)$ & $\alpha ^2$\\
\hline
$1$ & $1$ & $2$ & $1^2+1$ & $1^2$\\
\hline
$1$ & $2$ & $3$ & $2^2-1$ & $2^2$\\
\hline
$2$ & $3$ & $5$ & $3^2+1$ & $3^2$\\
\hline
$3$ & $5$ & $8$ & $5^2-1$ & $5^2$\\
\hline
$5$ & $8$ & $13$ & $8^2+1$ & $8^2$\\
\hline
$8$ & $13$ & $21$ & $13^2-1$ & $13^2$\\
\hline
$13$ & $21$ & $34$ & $21^2+1$ & $21^2$\\
\hline
$21$ & $34$ & $55$ & $34^2-1$ & $34^2$\\
\hline
$34$ & $55$ & $89$ & $55^2+1$ & $55^2$\\
\hline
$55$ & $89$ & $144$ & $89^2-1$ & $89^2$\\
\hline
$89$ & $144$ & $233$ & $144^2+1$ & $144^2$\\
\hline
$144$ & $233$ & $377$ & $233^2-1$ & $233^2$\\
\hline
$233$ & $377$ & $610$ & $377^2+1$ & $377^2$\\
\hline
$377$ & $610$ & $987$ & $610^2-1$ & $610^2$\\
\hline
$610$ & $987$ & $1597$ & $987^2+1$ & $987^2$\\
\hline
$987$ & $1597$ & $2584$ & $1597^2-1$ & $1597^2$\\
\hline
\end{tabular}
}

\bigskip

If, as I think, the School has really tried to find these values
of $\beta$ and $\alpha$
then they have all noticed the peculiarity
of the numbers in the table.
The Fibonacci numbers are
in the first, second and third column
and,
in addition, the square of the Fibonacci numbers
are in the fifth column
while the fourth column contains alternately
the predecessor and the successor of these squares,
see \cite{SEABM}.

Now, let $i\ge 0$ and $F_{i}$
the $i^{th}$ Fibonacci number.
{\it Does there exist an integer $\alpha \ge \beta$ such that
the difference between the greatest and the smallest
of the numbers $F_{i}(F_{i}+\alpha )$ and $\alpha ^2$
assumes the value 1?} Sure, it exists.
The table shows
that, for each $i$, $1\le F_i\le 1000$,
the required number 
$\alpha$
is exactly
$F_{i+1}$
and
$F_{i}(F_{i}+F_{i+1})-F_{i+1}^2=(-1)^{i}$.
Being $F_{i}+F_{i+1}=F_{i+2}$,
this equality becomes
$F_{i}F_{i+2}-F_{i+1}^2=(-1)^{i}$
and, as it is well-known,
the following lemma
holds (see for instance \cite{Irem}).

\begin{lemma}\label{Cassini identity}
{\bf Cassini identity.}
For each non negative integer $i$
and for each Fibonacci number $F_{i}$
the following equality holds

\centerline{$F_{i}F_{i+2}-F_{i+1}^2=(-1)^{i}$.}
\end{lemma}

As we have seen before,
step by step the School
has picked up new Fibonacci numbers.
Each new one discovered
corresponded to a more accurate (but not exact!)
measurement of the side and diagonal
of the regular pentagon. In this sense,
the School has discovered and proved the equality
$\lim_{n \rightarrow \infty}\frac{F_{n+1}}{F_{n}}=\Phi$,
certainly not in the very precise form
of the current modern epsilon-delta definition
that it has today,
but surely in the sense
that the difference
$\Phi -\frac{F_{n+1}}{F_{n}}$
became ever smaller
and smaller.

\section*{Cassini identity and characterization of Fibonacci numbers}

We introduce a definition
which will be crucial in the rest of the paper.

\begin{definition}\label{Definition}
{\it Let $\beta \ge 1$ an integer.
When there exists
an integer $\alpha $, $\alpha\ge\beta$,
such that, for some non-negative integer $\gamma$,
the equality

\centerline{$\beta (\beta +\alpha )-\alpha ^2=(-1)^{\gamma }$}

\noindent holds, then we say that $\beta$
is a {\rm Hippasus number}
and that $\alpha$
is a {\rm Hippasus successor} of $\beta$.
}
\end{definition}

For the aims of this paper, using the previous definition
\footnote{
This terminology
seems suitable.
Tradition, see \cite{von},
attributes to Hippasus
the discovery of incommensurability
and our thesis is the following:
the discoveries of incommensurability
and of a particular class of numbers
came simultaneously,
see \cite{SEABM}.
So these numbers that we show here
to be Fibonacci numbers
can provisionally be called
{\it Hippasus numbers}.},
we can obtain a more suitable reformulation
of the {\it Cassini identity}
\ref{Cassini identity}:

\begin{proposition}\label{FibonacciisHippasus1}
For each $i\ge 0$ the Fibonacci number $F_i$ is
a Hippasus number and $F_{i+i}$ is a
Hippasus successor of it.
\end{proposition}

The following Lemma \ref{FibonacciisHippasus2}
offers an even more precise reformulation
of the {\it Cassini identity}
\ref{Cassini identity}.
In order to prove Proposition
\ref{HippasusisFibonacci}
we need several lemmas.

\begin{lemma}\label{Hippasus1}
The number $1$ is a Hippasus number
and $1$ itself is one of its Hippasus successor.
\end{lemma}

\begin{proof}
The equality $1(1+1)-1^2=1$ holds.
\end{proof}

\begin{lemma}\label{Hippasus2}
The number $1$ has also $2$
as a Hippasus successor.
\end{lemma}

\begin{proof}
The equality $1(1+2)-2^2=-1$ holds.
\end{proof}

\begin{lemma}\label{Hippasus3}
No positive integer different from $1$ and $2$
is a Hippasus successor of $1$.
\end{lemma}

\begin{proof}
For $n>2$,
we have
$1(1+n)-n^2\le -5$.
\end{proof}

\begin{lemma}\label{Hippasus4}
If $\beta >1$ is a Hippasus number and $\alpha$
is one of its Hippasus successors then $\alpha >\beta$.
\end{lemma}

\begin{proof}
For $\beta >1$, the equality
$\beta (\beta +\beta )-(\beta )^2=(-1)^{\gamma}$
is impossible for each integer $\gamma$.
So if $\alpha $ exist we must have
$\alpha >\beta$.
\end{proof}

\begin{lemma}\label{Hippasus5}
{\it A Hippasus number
greater than $1$
has a unique
Hippasus successor.}
\end{lemma}

\begin{proof}
Let $\beta >1$ a Hippasus number
and $\alpha$ and $\alpha '$,
$\alpha \not= \alpha '$,
both Hippasus successors of $\beta$.
By the previous Lemma,
we have $\alpha >\beta$
and $\alpha '>\beta$.

Without loss of generality, suppose
$\alpha < \alpha '$.
There exists $\delta >0$
and
$\gamma ,\gamma '$ non negative integers such that
$\alpha '=\alpha + \delta$,
$\beta (\beta +\alpha )-\alpha ^2=(-1)^{\gamma }$
and
$\beta (\beta +\alpha +\delta)-(\alpha +\delta)^2=
(-1)^{\gamma '}$.
Now,

\centerline{$\beta (\beta +\alpha +\delta)-(\alpha +\delta)^2=$}

\centerline{$\beta (\beta +\alpha )+\beta\delta -(\alpha ^2+2\alpha\delta +\delta^2)=$}

\centerline{$(\beta (\beta +\alpha )
-\alpha ^2)+(\beta\delta -2\alpha\delta -\delta^2)=$}

\centerline{$(-1)^{\gamma }-\delta(-\beta +2\alpha +\delta )=$}

\centerline{$(-1)^{\gamma }-\delta((\alpha-\beta )+\alpha +\delta )$.}

Being
$\alpha \ge 3$ (as $\alpha >\beta \ge 2$),
$\alpha -\beta \ge 1$ (as $\alpha >\beta$)
and
$\delta \ge 1$ (as $\alpha '>\alpha$),
we have
$(\alpha-\beta )+\alpha +\delta \ge 5$
and
$- \delta ((\alpha-\beta )+\alpha +\delta )\le -5$.
So

\centerline{$\beta (\beta +\alpha +\delta)-(\alpha +\delta)^2
= (-1)^{\gamma }-5<(-1)^{\gamma '}$}

\noindent and $\alpha '=\alpha +\delta $ cannot
be a Hippasus successor of
$\beta$. Contradiction.
Then two different integers $\alpha , \alpha '$
cannot be both Hippasus successors of the same $\beta$.
\end{proof}

\bigskip

So, with the exception of 1
(that is, in a sense, {\it ambiguous})
any other Hippasus number $\beta$
has a {\it unique} Hippasus successor $\alpha$
that is strictly greater than $\beta$.

Now, we can precise Proposition \ref{FibonacciisHippasus1}

\begin{proposition}\label{FibonacciisHippasus2}
For the Fibonacci numbers
the following statements hold:

i) $F_{0}=1$ is an Hippasus number and
$F_{1}=1$ is an Hippasus successor of it,

ii) $F_{1}=1$ is an Hippasus number and
$F_{2}=2$ is an Hippasus successor of it,

iii) for each $i>1$, $F_{i}$ is an Hippasus number and
$F_{i+1}$ is its {\bf unique} Hippasus successor.
\end{proposition}

\begin{proof}
i) follows by Lemma \ref{Hippasus1}, 
ii) follows by Lemma \ref{Hippasus2}
and finally, as for $i>1$ we have $F_{i}\ge 2$, 
iii) follows by Proposition
\ref{FibonacciisHippasus1}
and
\ref{Hippasus5}.
\end{proof}

\begin{lemma}\label{Hippasus6}
{\it Let $\beta$ be a Hippasus number
and $\alpha$ be a Hippasus successor of $\beta$.
Then $\alpha -\beta\le \beta$.
}
\end{lemma}

\begin{proof}
It is a trivial verification if $\beta =1$ and $\alpha =1$
and
if $\beta =1$ and $\alpha =2$.
So, let $\beta >1$.
We know that,
for some $\gamma \ge 0$,
$\beta (\beta +\alpha )-\alpha ^2=(-1)^{\gamma}$.
By way of contradiction, suppose $\alpha -\beta > \beta$.
We have
$\beta (\beta +\alpha )-\alpha ^2=
(-1)[(\alpha -\beta )\alpha -\beta ^2]=$
$(-1)[(\alpha -\beta )((\alpha -\beta )+\beta) -\beta ^2]<
(-1)[\beta (\beta +\beta) -\beta ^2]=
-\beta ^2 \le -4$.
Contradiction. So, in any case, we have $\alpha -\beta\le \beta$.
\end{proof}

In some sense $0$ is a ``Hippasus number''
having 1 as one of its Hippasus successors
(indeed we have $0(0+1)-1=-1$)
but by our choice, a Hippasus number
must be positive, see Definition \ref{Definition}.
For this reason in the next lemma
we add the condition $\alpha >\beta$
with which we exclude
the case $\beta =1$ and $\alpha =1$.

\begin{lemma}\label{Hippasus7}
Let $\beta$ be a Hippasus number
and $\alpha$ be a Hippasus successor of $\beta$
with $\alpha >\beta$.
Then $\alpha -\beta$ is a Hippasus number
and
$\beta$ is a Hippasus successor of $\alpha -\beta$.
\end{lemma}

\begin{proof}
By Lemma \ref{Hippasus6}
we have $0<\alpha -\beta \le \beta$.
Moreover, we know that for some
$\gamma$ we have
$\beta (\beta +\alpha )-\alpha ^2=(-1)^{\gamma}$.
So,
$(\alpha -\beta )((\alpha -\beta )+\beta) -\beta ^2=
(\alpha -\beta )\alpha -\beta ^2=
(-1)[\beta (\beta +\alpha )-\alpha ^2]=(-1)^{\gamma +1}$
that exactly says that
$\alpha -\beta$ is a Hippasus number
and
$\beta$ is a Hippasus successor of $\alpha -\beta$.
\end{proof}

\begin{lemma}\label{Hippasus8}
{\it Let $\beta \ge 1$ be a Hippasus number
and $\alpha$ a Hippasus successor of $\beta$.
If $\alpha -\beta =\beta $ then
$\alpha -\beta =1$,
$\beta =1$
and
$\alpha=2$.
}
\end{lemma}

\begin{proof}
Consider three cases:

\centerline{a) $\beta =1$, $\alpha =1$;
\,\,\,\,
b) $\beta =1$, $\alpha =2$
\,\,\,\,
and
\,\,\,\,
c) $\beta >1$.}

We have $\alpha =2\beta$.
Case a): $\alpha =2\beta$ is not true.
Case c):
by Lemma \ref{Hippasus7},
for some $\gamma$
the equality
$(\alpha -\beta )(\alpha -\beta +\beta )-\beta ^2=(-1)^{\gamma}$
must hold,
i.e.,
$-\beta ^2=(-1)^{\gamma}$
must hold and
this contradicts $\beta >1$.
So, it remains only case b)
in which the statement trivially holds.
\end{proof}

Now, we are ready to prove the following proposition
of which we present two proofs.

\begin{proposition}\label{HippasusisFibonacci}
Any Hippasus number is a Fibonacci number.
\end{proposition}

\begin{proof}
Let $\beta$ be a Hippasus number
and let $\alpha$ be a Hippasus successor of it.
If $\beta =1$ and $\alpha =1$
then $\beta$ is a Fibonacci number.
If $\beta =1$ and $\alpha =2$
then $\beta$ is a Fibonacci number too.
(The set of Hippasus numbers
contain two times the value $1$,
see Lemma \ref{Hippasus1} and \ref{Hippasus2},
as well as the sequence of Fibonacci numbers.)

So, we have to prove that
a Hippasus number greater than $1$
is a Fibonacci number.
Let $\beta$ be such a number.
We know, by Definition \ref{Definition}, that $\beta$
has a Hippasus successor $\alpha$
and, being $\beta >1$, we also know
that $\alpha >\beta$
(by Lemma \ref{Hippasus4})
and that $\alpha$ is unique
(by Lemma \ref{Hippasus5}).

We know, by Lemma \ref{Hippasus7},
that $\alpha -\beta$ is a Hippasus number
and that $\beta $ is a Hippasus successor of $\alpha -\beta$.
By Lemma \ref{Hippasus6}
we have that
$\alpha -\beta \le \beta $,
i.e.,
there are two possibilities

\centerline{
$\alpha -\beta = \beta$
\,\,\,\,\,\,and\,\,\,\,\,\,
$\alpha -\beta < \beta$.}

If $\alpha -\beta =\beta$,
then by Lemma \ref{Hippasus8},
$\beta =1$. Contradiction.

So we must have
$\alpha -\beta < \beta$.
Put $\beta =\beta _1$ and $\alpha -\beta =\beta _2$.

It may
happen that
$\beta _1 -\beta _2 < \beta  _2$.
Put $\beta  _3=\beta _1 -\beta _2$.

It may similarly happen that
$\beta _2 -\beta _3 < \beta  _3$.
Put $\beta  _4=\beta _2 -\beta _3$.

And so on indefinitely.

In principle, we thus have two possibilities:

-either, for each positive integer $k$,
after the selection
of the integer $\beta _{k}$
we select $\beta _{k+1}$
with $\beta _{k+1}<\beta _{k}$;

-either the process of selection
of $\beta _{k+1}$ strictly smaller of $\beta _{k}$
will fail at a certain stage.

Let us take these two possibilities in turn
\footnote{
Here we try to imitate
a clear, elegant and powerful
model of exposition
that
{\it Ramsey}
presented in
\cite{Ramsey1930}.}.

By Pythagorean Proposition
\ref{Pythagorean Proposition}
({\it an infinite strictly decreasing sequence of positive
integers cannot exist})
the first possibility cannot happen.
So, the process of selection
of $\beta _{k+1}$ strictly smaller of $\beta _{k}$
will fail at a certain stage when,
for a given integer, say $i$,
$\beta _{i+1}=\beta _{i}$.

So, we suppose that we have selected
$\beta _{1}$,
$\beta _{2}$,
$\dots$,
$\beta _{i-2}$,
$\beta _{i-1}$,
$\beta _{i}$,
$\beta _{i+1}$
with 
$\alpha -\beta =\alpha -\beta _{1}<\beta _{1}$,
$\beta _{1}-\beta _{2}=\beta _{3}<\beta _{2}$,
$\beta _{2}-\beta _{3}=\beta _{4}<\beta _{3}$,
$\dots$,
$\beta _{i-2}-\beta _{i-1}=\beta _{i}<\beta _{i-1}$
and
$\beta _{i-1}-\beta _{i}=\beta _{i+1}=\beta _{i}$.

By hypothesis
$\beta =\beta _{1}$ is a Hippasus number
and 
$\beta _{2}$,
$\dots$,
$\beta _{i-2}$,
$\beta _{i-1}$,
$\beta _{i}$,
$\beta _{i+1}$
are all Hippasus numbers by Lemma \ref{Hippasus7}.
Moreover,
again by Lemma \ref{Hippasus7},
$\beta _{i}$ is a successor of $\beta _{i+1}$,
$\beta _{i-1}$ is a successor of $\beta _{i}$,
$\dots$,
$\beta _{1}$ is a successor of $\beta _{2}$,
$\alpha$ is a successor of $\beta =\beta _{1}$.

Considering $\beta _{i-1}-\beta _{i}=\beta _{i+1}=\beta _{i}$,
by Lemma \ref{Hippasus8}, we have:

\centerline{$\beta _{i+1}=1=F_{0}$,}

\centerline{$\beta _{i}=1=F_{1}$,}

\centerline{$\beta _{i-1}=2=F_{2}$.}

By construction $\beta _{i-1}=2=F_{2}$
has a unique Hippasus successor that is $\beta _{i-2}$
but, as the Fibonacci number $F_{2}$ has a unique
Hippasus successor that is $F_{3}$ (see Lemma \ref{FibonacciisHippasus2}),
we have that

\centerline{$\beta _{i-2}=3=F_{3}$.}

Similarly,

\centerline{$\beta _{i-3}=5=F_{4}$,}

\centerline{$\beta _{i-4}=8=F_{5}$,}

\centerline{$\dots \dots \dots $,}

\noindent and, continuing in this way,

\centerline{$\dots \dots \dots $,}

\centerline{$\beta _{3}=F_{i-2}$,}

\centerline{$\beta _{2}=F_{i-1}$,}

\centerline{$\beta _{1}=F_{i}$.}

\end{proof}

A second proof could be the following.
By way of contradiction,
suppose that the set of Hippasus numbers
which are not Fibonacci numbers is non empty.
By the {\it minimum principle} this set admits
a minimum element, say $\beta$.
Necessarily, $\beta$ is strictly greater than $2$
and has a unique
Hippasus successor, say $\alpha $.
Consider $\alpha - \beta$
that, by Lemma \ref{Hippasus7},
is a Hippasus number.
If $\alpha - \beta =\beta$
then, by Lemma \ref{Hippasus8},
$\beta =1$ that is a Fibonacci number.
Contradiction.
If $\alpha - \beta <\beta$
then, by Lemma \ref{Hippasus7},
$\alpha - \beta$
is a Hippasus number and strictly smaller than $\beta$.
Contradiction too.

The second proof,
that uses the {\it minimum principle},
is shorter than the first one,
which we prefer as it uses explicitly
the {\it Pythagorean Proposition}
\ref{Pythagorean Proposition}.

Proposition \ref{FibonacciisHippasus1}
and
Proposition \ref{HippasusisFibonacci}
imply the following

\begin{proposition}\label{Proposition3}
A positive integer is a Hippasus number if, and only if, it is a Fibonacci number.
\end{proposition}

By our previous results we are convinced that
the relations between the pythagorean
{\it equality}
$b(b+a)-a^2=0$
and the {\it Cassini Identity}
$\beta (\beta +\alpha )-\alpha ^2=(-1)^{\gamma}$
are really very strict.
At least in our thesis,
the School,
that discovered the first equality,
hardly could have ignored the second one.
In other terms, when the School found a Hippasus number then
the same School simultaneusly  found
a Fibonacci number,
because no other number could have be found.
In order to add another argument to our previous ones,
we prove the following proposition.

\begin{proposition}\label{alphabeta}
Let $\beta$ be a Hippasus number
and $\alpha$ be a Hippasus successor of $\beta$.
Then $\alpha$ is a Hippasus number
and
$\alpha + \beta$ is a Hippasus successor of $\alpha$.
\end{proposition}

\begin{proof}
For some $\gamma$ we have
$(-1)^{\gamma}=
\beta (\beta +\alpha )-\alpha ^2$.
So
$\alpha (\alpha +\alpha +\beta )-(\alpha +\beta)^2=
\alpha ^2 +\alpha (\alpha +\beta )-(\alpha +\beta)^2=
\alpha ^2 -\alpha \beta  -\beta^2=
(-1)(-\alpha ^2 +\alpha \beta  +\beta^2)=
(-1)(\beta (\beta +\alpha )-\alpha ^2)=(-1)^{\gamma +1}$
i.e.
$\alpha$ is a Hippasus number
and $\alpha +\beta$ is a Hippasus successor of it.
\end{proof}

\begin{corollary}\label{ab}
If $a$ is a Hippasus number and $b$
is its Hippasus successor
then $a+b$ is a Hippasus number.
\end{corollary}

Corollary \ref{ab}
certifies that the laws of formation of Fibonacci numbers
and of Hippasus numbers are the same!
Much better, the Fibonacci law $F_n+F_{n+1}=F_{n+2}$
rediscovers the Pythagorean law given in the previous
Corollary \ref{ab}.
Moreover,
the Definition \ref{Definition}
of Hippasus numbers is
operational and allows us to find Hippasus
numbers one after the other.

The Wasteel result of next section
is just a criterion to decide if two integers
are consecutive Fibonacci numbers.

\section*{With Fibonacci numbers the surprises never end}

Dickson recalls in \cite{Dickson}
the following result of Wasteels, proved in \cite{wasteels}.

\begin{proposition}\label{Proposition4}
Two positive integers $x$ and $y$
for which $y^2 -xy -x^2$
equals $+1$ or $-1$
are consecutive terms of the series of Fibonacci.
\end{proposition}

Matiyasevich in
\cite{Matiyasevich1}
with reference to the result of Wasteels
says:
{\it The fact that successive Fibonacci numbers
give the solution of Eq. (25) was
presented by Jean-Dominique Cassini
to the Academie Royale des Sciences as
long ago as 1680. It can be proved
by a trivial induction. At the same time the
stronger fact that Eq. (25) is characteristic
of the Fibonacci numbers is somehow
not given in standard textbooks.
The induction required to prove the converse is
less obvious, and that fact seems
to be the reason for the inclusion of the problem
of inverting Cassini’s identity as Exercise 6.44 in Concrete Mathematics by Ronald Graham, Donald Knuth, and Oren Patashnik [13]. As the original source of this problem
the authors cite my paper [21], but I have always suspected
that such a simple and fundamental
fact must have been discovered long before
me. This suspicion turned out to be justified:
I have recently found a paper of
M.Wasteels [41] published in 1902
in the obscure journal Mathesis.
\footnote
{In this citation Eq. (25) is the Cassini identity
$F_{i}F_{i+2}-F_{i+1}^2=(-1)^{i}$.
Paper [13] corresponds to \cite{GKP} here,
paper [21] is the fundamental and historical paper
of Matiyasevich (here \cite{Matiyasevich0})
and paper of Wasteels [41] is \cite{wasteels} here.}
}

A {\it pentagon}
on a {\it portale} of ``{\it Duomo di Prato}''
refers to {\it Fibonacci numbers}
\footnote{Recently, the Fibonacci numbers
have been rediscovered in a tarsia
of the Church of San Nicola in Pisa
(see Armienti
\cite{Armienti}
and Albano
\cite{Albano}).
}
and a {\it octagon} on the same {\it portale}
seems to have a reference to a singular construction
of an octagon that uses Fibonacci numbers!
This octagon is {\it not}
regular but very impressively
similar to a regular octagon:
we design two concentric circles
having diameters $F_n$ and $F_ {n + 2}$,
the two horizontal straight line tangent to the inner circle
and
the two vertical straight line tangent to the same inner circle.
These four lines
cut the larger circle
into 8 points.
We denote by $P_n$ and $Q_n$
the two of them having
the following coordinates
and lying in the first quadrant:

\centerline{
$P_n=\Big( \frac{F_n}{2}
,
\sqrt {\Big(\frac{F_{n+2}}{2}\Big)^{2}-
\Big(\frac{F_n}{2}\Big)^{2} }\Big)$
\,
,
\,
$Q_n=\Big( \sqrt {\Big(\frac{F_{n+2}}{2}\Big)^{2}-
\Big(\frac{F_n}{2}\Big)^{2} }\Big),
\frac{F_n}{2}\Big)$
.}

They are the extremes
of one of the eigth sides of our octagon.
We note that their distance $d_n$ is
$\sqrt{2}
\Big[\sqrt {{\Big(\frac{F_{n+2}}{2}\Big)^{2}-
\Big(\frac{F_n}{2}\Big)^{2} }}-\frac{F_n}{2}\Big]$.
We also denote by $e_n$
the side of the regular octagon
inscribed in the circle of diameter $F_{n+2}$.
We have that:

-the value
$\frac{d_{n}}{F_{n}}$
tends to the limit
$
\frac{\sqrt{2}}{2}
\Big[\sqrt {\Phi ^{4}-1}-1\Big]$,
i.e. about 1.00375

-the value
$\frac{d_{n}}{e_{n}}$
tends to the limit
$
\frac{\sqrt{2}}{\sqrt{2-\sqrt{2}}}
\Big[\sqrt {1-\Phi ^{-4}}-\Phi ^{-2}\Big]
$, i.e. about 1.00187

-the value $\frac{e_{n}}{F_{n}}$
tends to the limit
$\frac{\sqrt{2-\sqrt{2}}}{2}\Phi ^2$,
i.e. about 1.00187.

It seems that the architech of the
``{\it Duomo di Prato}''
was Carboncettus marmorarius see
\cite{Cerretelli}
and
\cite{Fantappie}.
For these reasons
one can speak about
Carboncettus octagon!

\section*{Acknowledgements}
I thank my brother Mario,
lecturer at the Educandato
Santissima Annunziata in Firenze,
for the program that permitted us
to build
the table of this article.
I thank also
Maurizio Aristodemo, 
Lorenzo Bussoli,
Luigi Barletti,
Gabriele Bianchi,
Gabriele Villari,
for a first reading of this paper.
I thank the Dipartimento di Matematica e Informatica
``Ulisse Dini''
for his hospitality
and
the Project Interomics of CNR for the financial support..

\bibliography{library}

\bibliographystyle{abbrv}

\end{document}